\documentclass[a4paper]{gtart}

\usepackage{amsmath, amssymb, latexsym, euscript, color, hyperref, aliascnt}

\usepackage{mathptmx, wrapfig}

\usepackage[inner=20mm, outer=20mm, textheight=235mm]{geometry}

\usepackage{booktabs}

\def\R{\mathbb{R}}

\def\C{\mathbb{C}}

\DeclareMathOperator{\Vol}{Volume}
\DeclareMathOperator{\rank}{rk}
\DeclareMathOperator\sign{sign}

\theoremstyle{plain}
\newtheorem{theorem}{Theorem}

\theoremstyle{definition}

\newtheorem*{definition*}{Definition}

\theoremstyle{remark}

\newtheorem{task}[theorem]{Task}
\newtheorem{question}[theorem]{Question}

\newaliascnt{lemma}{theorem}

\aliascntresetthe{lemma}
\newaliascnt{proposition}{theorem}
\newtheorem{proposition}[proposition]{Proposition}
\aliascntresetthe{proposition}
\newaliascnt{corollary}{theorem}
\newtheorem{corollary}[corollary]{Corollary}
\aliascntresetthe{corollary}
\newaliascnt{observation}{theorem}

\aliascntresetthe{observation}

\numberwithin{equation}{section}

\begin{document}

\title{Reflections on trisection genus}
\author{Michelle Chu and Stephan Tillmann}

\begin{abstract} 
The Heegaard genus of a 3--manifold, as well as the growth of Heegaard genus in its finite sheeted cover spaces, has extensively been studied in terms of algebraic, geometric and topological properties of the 3--manifold. This note shows that analogous results concerning the trisection genus of a smooth, orientable 4--manifold have more general answers than their counterparts for 3--manifolds. In the case of hyperbolic 4--manifolds, upper and lower bounds are given in terms of volume and a trisection of the Davis manifold is described.
\end{abstract}

\makeshorttitle


\subsection*{Bounds on trisection genus}

Gay and Kirby's construction of a \emph{trisection} for arbitrary smooth, oriented closed 4--manifolds~\cite{GK} defines a decomposition of the 4--manifold into three $4$-dimensional 1--handlebodies glued along their boundaries in the following way: Each handlebody is a boundary connected sum of copies of $S^1 \times B^3,$ and has boundary a connected sum of copies of $S^1 \times S^2$ (here, $B^i$ denotes the $i$-dimensional ball and $S^j$ denotes the $j$-dimensional sphere). The triple intersection of the $4$-dimensional 1--handlebodies is a closed orientable surface $\Sigma$, called the {\em central surface}, which divides each of their boundaries into two $3$--dimensional 1--handlebodies (and hence is a Heegaard surface). These $3$--dimensional 1--handlebodies are precisely the intersections of pairs of the 4--dimensional 1--handlebodies.

A trisection naturally gives rise to a quadruple of non-negative integers $(g; g_0, g_1, g_2)$, encoding the genus $g$ of the central surface $\Sigma$ and the genera $g_0,$ $g_1,$ and  $g_2$ of the three 4--dimensional 1--handlebodies. We allow the genera to be distinct following~\cite{Meier-classification-2016, Rubinstein-multisections-2016}.
The \emph{trisection genus} of $M,$ denoted $g(M),$ is the minimal genus of a central surface in any trisection of $M.$ A trisection with $g(\Sigma) = g(M)$ is called a \emph{minimal genus trisection}. 

If $M$ has a $(g; g_0, g_1, g_2)$--trisection, then $\chi(M) = 2 + g - g_0 - g_1 - g_2.$ Since the fundamental group of each 4--dimensional 1--handlebody surjects onto the fundamental group of $M,$ we have $g_k \ge \beta_1(M).$ Combining the above equation with $\chi(M) = 2 - 2 \beta_1(M) + \beta_2(M)$ therefore gives
\[ g(M) \ge \beta_1(M) + \beta_2(M).\]
This is an equality for the six manifolds with trisection genus at most two~\cite{MZ-standard}, and for all standard simply connected 4--manifolds~\cite{Spreer-trisection-2018}. Castro and Ozbagci~\cite{Castro-trisections-2017} constructed trisections of genus $2g+2$ of the twisted bundles $S_g \widetilde{\times} S^2,$ where $S_g$ is a closed, connected, oriented surface of genus $g.$ By the above inequality, these are all minimal and so $g(S_g \widetilde{\times} S^2) = 2g+2.$
In the case of $S_g \times S^2,$ Castro and Ozbagci~\cite{Castro-trisections-2017} constructed trisections of genus $2g+5.$ Hence the lower bound of $2g+2$ is at most off by 3. 

We have $g_k \ge \rank \pi_1(M) \ge \beta_1(M),$ where $\rank(\pi_1(M))$ is the minimal number of generators of the fundamental group of $M$. Hence lower bounds on the rank give lower bounds on trisection genus, and, conversely, upper bounds on the trisection genus give upper bounds on the rank. Moreover,  $g(M) \ge \chi(M) - 2 + 3 \rank(\pi_1(M)).$

\begin{question}
Given any finitely presented group $G,$ is there a  smooth, oriented closed 4--manifold $M$ with $\pi_1(M) = G$ such that $g(M) = \chi(M) - 2 + 3 \rank(G)$?
\end{question}

The question has a positive answer for surface groups (using the above examples of $S_g \widetilde{\times} S^2$) and the infinite cyclic group (by virtue of $S^1\times S^3$). It thus has a positive answer for free products of these groups (in particular, finitely generated free groups), by taking connected sums of these manifolds and application of Grushko's theorem~\cite{Grushko-bases-1940, Neumann-number-1943}.

We now turn to an upper bound.
Let $\sigma(M)$ be the minimal number of 4--simplices in a (possibly singular) triangulation of $M.$ We call $\sigma(M)$  the \emph{triangulation complexity} of $M.$ Then the upper bound on trisection genus recently given in \cite{Bell-computing-2017} using triangulations gives:
\[ 60\ \sigma(M) \ge g(M).\]

This note gives some simple consequences of these upper and lower bounds. Our starting point is the following, which at once gives upper and lower bounds on trisection genus, \emph{and} a lower bound on the triangulation complexity of a 4--manifold.

\begin{proposition}\label{pro:main bound}
Let $M$ be a smooth orientable 4--manifold and assume that $M$ is not diffeomorphic to $S^4.$ Then
\[60 \sigma(M) \ge g(M) \ge \frac{1}{3} | \chi(M)|.\]
\end{proposition}

The proof is simple. We have $\chi(M) = 2 + g - g_0-g_1 - g_2.$ Hence $g(M)\ge \chi(M)-2.$ In particular, the lower bound holds when $\chi(M)\ge 3.$ If $\chi(M)>0$ and $\beta_1(M)>0,$ then $g(M)\ge \chi(M).$ Now suppose $1\le \chi(M) \le 2,$ $\beta_1(M)=0$ and $M \neq S^4.$ 
The classification of trisections of genus $2$ due to Meier and Zupan~\cite{MZ-standard}, then implies $g(M)\ge 3 \ge \chi(M).$  So the lower bound holds in the case of positive Euler characteristic. Moreover, if $\C P^2$ is excluded from the statement, one may replace $1/3$ by $1/2.$

Combining $\chi(M) = 2 - 2 \beta_1(M) + \beta_2(M)$ with the bound $g(M) \ge \beta_1(M)$ gives $g(M)\geq 1-\frac{1}{2}\chi(M)$.
So if $\chi(M)\le 0,$ then $g(M)\ge \frac{1}{2}|\chi(M)| + 1.$ This is sharp for $S^1\times S^3$, where $\chi(S^1\times S^3)=0$ and $g(S^1\times S^3)=1.$ \qed


\subsection*{Stable trisection genus}

We now turn to finite sheeted covers. In particular, we are interested in the case where there are finite covers of arbitrarily large degree, but our results also apply when there are only finitely many such covers. If $M$ is hyperbolic, then residual finiteness implies that there are finite covers of arbitrarily large degree. If $\chi(M) < 0,$ then $\beta_1(M)>0,$ and hence there are finite cyclic covers of any degree. Since the only manifold covered by $S^4$ is the non-orientable manifold $\R P^4$~\cite[Section 12.1]{hillman-four-2002}, we have:

\begin{corollary}\label{cor:cover bounds}
Let $M$ be a smooth orientable 4--manifold and assume that $M$ is not diffeomorphic to $S^4.$
Suppose $N \to M$ is a finite sheeted cover of degree $d.$ Then $g(N)$ is $O(d).$ Specifically,
\[60 \sigma(M) \ d \ge g(N) \ge \frac{1}{3} | \chi(M)|\ d.\]
Moreover, if $\chi(M)\neq 0,$ then $g(N)$ is $\Theta(d).$
\end{corollary}

Another immediate corollary is in terms of the characteristic functions of Milnor and Thurston~\cite{Milnor-characteristic-1977}. Denote 
\[\sigma_\infty(M) = \inf \Big\{\ \frac{\sigma(N)}{d} \ \mid N \to M \text{ a degree $d$ cover} \Big\}\]
and define the \emph{stable trisection genus} by
\[g_\infty(M) = \inf \Big\{\ \frac{g(N)}{d} \ \mid N \to M \text{ a degree $d$ cover} \Big\}.\]
For instance $g_\infty(S^1\times S^3) = 0$, $g_\infty(T^4) = 0,$ $g_\infty(S_g\times S^2) = 2g-2 = - \chi(S_g) = g_\infty(S_g\widetilde{\times} S^2)$ if $g\ge 1.$

The definition of stable trisection genus and \autoref{pro:main bound} imply:

\begin{corollary}
Let $M$ be a smooth orientable 4--manifold and  assume that $M$ is not diffeomorphic to $S^4.$ Then
\[60\ \sigma_\infty(M) \ge g_\infty(M) \ge \frac{1}{3} | \chi(M)|. \]
\end{corollary}

Betti numbers grow at most linearly in the degree, since the homology is carried by a lifted triangulation. Moreover, the equality $\chi(M) = 2 - 2 \beta_1(M) + \beta_2(M)$ implies that $\beta_2(M)$ must grow linearly if $\chi(M) > 0,$ and $\beta_1(M)$ must grow linearly if $\chi(M) < 0.$
The question of whether the growth is linear or sublinear for the other Betti number is governed by the $L^2$--Betti number $b_k^{(2)}(\widetilde{M}; \pi_1(M))$ due to work of L\"uck~\cite{Luck}.
Sublinear growth of $\beta_1$ and linear growth of $\beta_2$ occurs in covers of hyperbolic 4--manifolds, due to positive Euler characteristic and the work of L\"uck~\cite{Luck}.
Linear growth of both $\beta_1$ and $\beta_2$ occurs in cyclic covers of $(S_g \times S^2)\#(S^1\times S^3),$ where one unwraps the $S^1$--factor.
Linear growth of $\beta_1$ and sublinear growth of $\beta_2$ occurs in covers of $S_g \times S^2,$ where $\beta_2$ has constant value $2.$


\subsection*{Geometric consequences}

In dimension four, Euler characteristic is related to signature (in case the manifold is Einstein), volume or injectivity radius (in case the manifold if hyperbolic). This gives simple applications of the previous results. For instance, the Gromov-Hitchin-Thorpe inequality~\cite{Gromov-volume-1982, Kotschik-gromov-1998} implies 

\begin{corollary}
Let $M$ be a closed Einstein 4--manifold not diffeomorphic to $S^4.$ Then 
$$g(M) \ge \frac{1}{2}|\sign(M)| + \frac{1}{7776\pi^2}||M||,$$
where $\sign(M)$ is the signature of the intersection form and $||M||$ is the Gromov norm.
\end{corollary}

For congruence covers of arithmetic hyperbolic 3--manifolds, Lackenby~\cite[Corollary 1.6]{Lackenby-heegaard-2006} bounds Heegaard genus in terms of volume. His lower bound is established in terms of the Cheeger constant, which is uniformly bounded below for congruence covers. The following counterpart for trisection genus is both more general and has an elementary proof. The following two observations give some answer to Problem 1.24 in \cite{AIM}.

\begin{corollary}
Let $M$ be a closed hyperbolic 4--manifold. Then there is a positive constant $C=C(M)$ such that for any finite cover $M_i\to M,$
$$C\ \Vol(M_i) \ge g(M_i) \ge \frac{3}{8 \pi^2} \Vol(M_i).$$
\end{corollary}

\begin{proof}
As noted in the proof of \autoref{pro:main bound}, $g(M_i)\ge \chi(M_i)-2$. Since $\chi(M_i)$ is positive and even and $g(M_i)\geq 3$, in fact $g(M_i)\ge \frac{1}{2}\chi(M_i)$. This together with $\Vol(M_i)=\frac{4}{3}\pi^2\chi(M_i)$ gives the lower bound.
The upper bound follows directly from \autoref{cor:cover bounds} by taking $C(M)=60\frac{\sigma(M)}{\Vol(M)}.$
\end{proof}

Since volume bounds injectivity radius, we have:

\begin{corollary}
Let $M$ be a closed hyperbolic 4--manifold and let $\mathrm{inj}(M)$ denote its injectivity radius. Then there is a universal positive constant $C'$ such that 
\[ g(M)\ge C'\cdot \mathrm{inj}(M).\]
\end{corollary}


\subsection*{The Davis manifold}

We conclude this paper by describing trisections of the Davis manifold~\cite{Davis-hyperbolic-1985}. The Davis manifold $M_D$ is obtained by identifying opposite pairs of faces of the 120-cell with dihedral angles $\frac{2\pi}{5}$ in hyperbolic space. The boundary of the 120-cell has a cell decomposition with 120 dodecahedra, 720 pentagons, 1200 edges and 600 vertices. The face identifications leave us in $M_D$ with 60 dodecahedra, 144 pentagons, 60 edges and 1 vertex. A natural triangulation of $M_D$ is obtained as follows. Place a vertex $v_4$ at the centre of the 120-cell, a vertex $v_3$ at the centre of a dodecahedral face, a vertex $v_2$ at the barycentre of a pentagonal face thereof, a vertex $v_1$ at the barycentre of an edge of this, and a vertex $v_0$ at a vertex of this edge. This gives a Coxeter 4--simplex, usually denoted $\Delta_3,$ and the 120-cell is tiled by $(120)^2=14,400$ 4--simplies that are copies of $\Delta_3.$ In particular, there is a group $\Gamma$ of order $(120)^2$ acting on $M_D$ with $M_D/\Gamma$ equal to the simplex orbifold with underlying space $\Delta_3$. Ratcliffe and Tschantz~\cite{Ratcliffe-davis-2001} computed $\beta_1(M_D) = 24$ and $\beta_2(M_D)=72.$ Whence 
\[ 864,000 = 60 \cdot (120)^2 \ge g(M_D) \ge 96,\]
using the apriori bounds. We now describe how in situations such as the Coxeter construction of the Davis manifold, this upper bounds can be greatly improved using the 
techniques of \cite{Bell-computing-2017, Rubinstein-multisections-2016}. 

The idea is to partition the vertices of the triangulation into three sets, such that each 4--simplex meets one of them in one vertex and each of the other two sets in two vertices.  This partition is then used to define a piece-wise linear map to the 4--simplex. It is shown in \cite{Bell-computing-2017, Rubinstein-multisections-2016} that applying certain 2--4 bistellar moves to such a tricoloured triangulation gives a new triangulation with the property that the pull-back of the natural cubulation of the 2--simplex defines a trisection of the 4--manifold. We describe this for the case of the Davis manifold, and refer the reader to the treatment in \cite{Bell-computing-2017} for more details.

Define a partition of the vertices of the 4--simplices by $S_0 = \Gamma \cdot v_0 \cup \Gamma \cdot v_1,$  $S_1 = \Gamma\cdot v_2$ and $S_2 = \Gamma \cdot v_3 \cup \Gamma \cdot v_4.$ Then the graphs $\Gamma_k$ spanned by $S_k$ in the 1--skeleton of the triangulation of $M_D$ have the following properties. The graph $\Gamma_1$ is the quotient of the 1--skeleton of the 120-cell and hence a bouquet of 60 circles. The graph $\Gamma_2$ is the dual 1--skeleton of the cellulation of $M_D$ arising from the 120-cell and hence also a bouquet of 60 circles. The graph $\Gamma_2$ consists of 144 isolated vertices. Each 4--simplex has two vertices in each $S_0$ and $S_2,$ and one in $S_1.$ Hence it meets a unique 4--simplex in the tetrahedral face with all vertices in $S_0 \cup S_1.$ This gives a decomposition of $M_D$ in double-4--simplices. As described in \cite[Construction 3]{Bell-computing-2017}, we now apply 2--4 bistellar moves to each of these double 4--simplices. This increases the number of pentachora to 28,800. The graphs $\Gamma_0$ and $\Gamma_1$ are not changed, and $\Gamma_2$ turns into a connected graph with 144 vertices and 7,200 edges. Hence $\Gamma_2$ is homotopic to a bouquet of 7,057 circles.

We now compute the Euler characteristic of the central surface $\Sigma$. We obtain one square for each pentachoron, hence there are 28,800 squares. The number of vertices of the surface equals the number of triangles in the triangulation that have vertices in all partition sets. It is not difficult to check that there are 14,400 such triangles. From this information, we compute $g(\Sigma) = 7,201.$ Whence
\[ 7,201 \ge g(M_D) \ge 96.\]
The above approach applies to any Coxeter type situation. In case of the Davis manifold, improvements can be made by choosing smaller triangulations of the 120-cell that still have the desired partition properties. Our current best upper bound, however, remains at 5621, and does not improve the current gap in magnitudes. It would be interesting to see whether greater improvements can be obtained for the known hyperbolic 4--manifolds arising from Coxeter constructions.


\textbf{Concluding remarks.}
The main challenge in obtaining lower bounds on trisection genus lies in dermining lower bounds on the genera of the 4--dimensional 1--handlebodies. Mostow rigidity, of course, implies that every algebraic or topological invariant of a hyperbolic 4--manifold is a geometric invariant, but there is a philosophical distinction. It is in this spirit that we formulate the following:

\begin{task}
Determine stronger lower bounds on trisection genus of hyperbolic 4--manifolds using the \emph{geometry} of the manifold.
\end{task}


\subsection*{Acknowledgements}

The authors thank David Gay and Jonathan Hillman for helpful comments.
This work is supported by ARC Future Fellowship FT170100316. 



\address{Michelle Chu\\Department of Mathematics, University of California, Santa Barbara, CA 93106, USA\\{mchu@math.ucsb.edu}\\----}

\address{Stephan Tillmann\\School of Mathematics and Statistics F07, The University of Sydney, NSW 2006 Australia\\{stephan.tillmann@sydney.edu.au}}

\Addresses


\begin{thebibliography}{99}
  
  
  \bibitem{AIM} AimPL: Trisections and low-dimensional topology, available at http://aimpl.org/trisections.   
  
  \bibitem{Bell-computing-2017}
Mark {Bell}, Joel {Hass}, J.~{Hyam Rubinstein}, and Stephan {Tillmann}:
\emph{Computing trisections of 4-manifolds}.
to appear in Proceedings of the National Academy of Sciences of the United States of America.
  {http://arxiv.org/abs/1711.02763}.
   
\bibitem{Castro-trisections-2017} Nickolas A. Castro, Burak Ozbagci: \emph{Trisections of 4-manifolds via Lefschetz fibrations}, 	arXiv:1705.09854
  
 \bibitem{Davis-hyperbolic-1985} Michael W. Davis: \emph{A hyperbolic 4-manifold.}
Proc. Amer. Math. Soc. 93 (1985), no. 2, 325--328. 
    
  \bibitem{GK}
David Gay and Robion Kirby: \emph{Trisecting 4-manifolds.}
 {\em Geom. Topol.}, 20(6):3097--3132, 2016.
    
  \bibitem{Gromov-volume-1982} Mikhael Gromov: \emph{Volume and bounded cohomology.} Publ. Math. I.H.E.S.
56
(1982),
5--99.

\bibitem{Grushko-bases-1940} Igor A. Grushko: \emph{On the bases of a free product of groups}. Matematicheskii Sbornik, vol 8 (1940), pp. 169--182.
  
  \bibitem{hillman-four-2002} Jonathan Hillman: \emph{Four-Manifolds, Geometries and Knots}, Geometry and Topology Monographs, vol. 5, Geometry and Topology Publications, December 2002.
  
\bibitem{Kotschik-gromov-1998}  Dieter Kotschick: \emph{On the Gromov-Hitchin-Thorpe inequality}. C. R. Acad. Sci. Paris
326 (1998), 727--731.
  
\bibitem{Lackenby-heegaard-2006} Marc Lackenby: \emph{Heegaard splittings, the virtually Haken conjecture and property $(\tau)$}, Invent. Math. 164 (2006), no. 2, 317--359.

\bibitem{Luck} Wolfgang L{\"u}ck: \emph{$L^2$-torsion and 3-manifolds} (Knoxville, TN, 1992), International Press, Cambridge, MA, 1994, pp. 75--107. 

\bibitem{Meier-classification-2016}
Jeffrey Meier, Trent Schirmer, and Alexander Zupan.
 Classification of trisections and the generalized property {R}
  conjecture.
 {\em Proc. Amer. Math. Soc.}, 144(11):4983--4997, 2016.

\bibitem{MZ-standard}
Jeffrey Meier and Alexander Zupan: \emph{Genus-two trisections are standard.}
 {\em Geom. Topol.}, 21(3):1583--1630, 2017.
 
\bibitem{Milnor-characteristic-1977}  John Milnor and William Thurston: \emph{Characteristic numbers of $3$-manifolds.} Enseignement Math. (2) 23 (1977), no. 3-4, 249--254.

\bibitem{Neumann-number-1943} Bernhard H. Neumann: \emph{On the number of generators of a free product.} J. London Math. Soc. 18, (1943). 12--20.

\bibitem{Ratcliffe-davis-2001} John G. Ratcliffe and Steven T. Tschantz: \emph{On the Davis hyperbolic 4-manifold.}
Topology Appl. 111 (2001), no. 3, 327--342. 

\bibitem{Rubinstein-multisections-2016}
J.~Hyam {Rubinstein} and Stephan Tillmann:
 \emph{Multisections of piecewise linear manifolds}.
to appear in Indiana University Mathematics Journal,
  {http://arxiv.org/abs/1602.03279} .
  
\bibitem{Spreer-trisection-2018} Jonathan Spreer and Stephan Tillmann: \emph{The trisection genus of the standard simply connected PL $4$--manifolds.} 34th International Symposium on Computational Geometry (SoCG 2018). In Leibniz International Proceedings in Informatics (LIPICS), vol. 99, pg. 71:1-71:13, 2018.
  
  

\end{thebibliography}
\end{document}